\documentclass[12pt]{amsproc}
\usepackage{amsmath,amsthm}
\usepackage{amsfonts,amssymb}
\newtheorem{thm}{Theorem}[section]

\newtheorem{lem}[thm]{Lemma}
\newtheorem{prop}[thm]{Proposition}

\theoremstyle{remark}
\newtheorem{rem}{Remark}[section]

\def\C{{\mathbb C}}

\def\N{{\mathbb N}}
\def\R{{\mathbb R}}

\def\1{\text{\bf {1}}}

\DeclareMathOperator{\csch}{cosech}

\begin{document}

\title[Mixed norm estimates for Hermite multipliers]
{Mixed norm estimates for Hermite multipliers}
\author{R. Lakshmi Lavanya and S. Thangavelu}

\address{Department of Mathematics\\ Indian Institute
of Science\\Bangalore-560 012}
\email{lakshmi12@math.iisc.ernet.in, veluma@math.iisc.ernet.in}

\date{\today}
\keywords{Hermite semigroup, Laguerre semigroup, Bochner-Riesz
means, Littlewood-Paley g-functions, Hermite multipliers}
\subjclass{Primary: 42C10, 47G40, 26A33. Secondary: 42B20, 42B35,
33C45.}

\begin{abstract}
In this article mixed norm estimates are obtained for some
integral operators, from which those for the Hermite semigroup and
the Bochner-Riesz means associated with the Hermite expansions are
deduced. Also, mixed norm estimates for the Littlewood-Paley
$g-$functions and $g^*-$functions for the Hermite expansions are
obtained, which lead to those for Hermite multipliers.
\end{abstract}

\maketitle

\section{Introduction}
The aim of this article is to study mixed norm estimates for
multiplier operators associated to Hermite expansions. Given a
function $\varphi$ defined on the set of all positive integers one
can define the operator
$$\varphi(H) = \sum\limits_{k=0}^{\infty} \ \varphi(2k+n) \ P_k$$
where $P_k$ are the spectral projections associated to the Hermite
operator $H=- \Delta +|x|^2$ on $\R^n.$ In \cite{Th} sufficient
conditions on $\varphi$ have been given so that $\varphi(H)$
defines a bounded linear operator on $L^p(\R^n),$ for
$1<p<\infty.$ In this paper we are interested in mixed norm
estimates for $\varphi(H).$ Let $L^{p,2}(\R^n)$ stand for the
space of functions $f(r\omega)$ on $\R^+\times S^{n-1}$ for which
$$\|f\|_{L^{p,2}(\R^n)} = \left( \int
\limits_0^\infty \left(\ \int\limits_{S^{n-1}} |f(r\omega)|^2
d\omega\right)^\frac{p}{2} r^{n-1} \ dr\right)^{\frac{1}{p}}$$ are
finite. We show that under some conditions on the multiplier
$\varphi,$ the mixed norm estimates
$$\|\varphi(H)f\|_{L^{p,2}(\R^n)} \leq C \|f\|_{L^{p,2}(\R^n)}$$
hold for all $1<p<\infty.$

This work is motivated by the recent paper \cite{CR} by Ciaurri
and Roncal where they have studied the boundedness of Riesz
transforms associated to the Hermite operator written in polar
coordinates. The Hermite operator $H$ admits a family of
eigenfunctions given by the multidimensional Hermite functions
$\Phi_\alpha(x).$ Written in polar coordinates the same operator
admits another family of eigenfunctions which are of the form
$g(r)Y(\omega), \ x=r\omega\in \R^n$ where $Y$ are spherical
harmonics and $g$ are suitable Laguerre functions. In terms of
these eigenfunctions one can define Riesz transforms for the
Hermite operator. In \cite{CR} the authors have proved the
boundedness of these Riesz transforms. It turns out that their
result is equivalent to mixed norm estimates for the standard
Riesz transforms $R_j = A_jH^{-\frac{1}{2}}, \ j=1,2,\cdots ,n,$
associated to Hermite expansions. Here $A_j =
\frac{\partial}{\partial x_j}+x_j$ are the annihilation operators
of quantum mechanics.

In view of the above observation it is natural to ask for mixed
norm estimates for certain operators associated to Hermite
expansions. In this paper we consider the Hermite semigroup,
Bochner-Riesz means and more general multiplier transformations
for the Hermite expansions. The connection between the two kinds
of expansions viz., the one in terms of the standard Hermite
functions and the other in terms of spherical harmonics times
Laguerre functions is provided by the Hecke-Bochner formula for
the Hermite expansion, see Theorem 3.4.1 in \cite{Th}. However,
our investigation of mixed norm estimates for Hermite multipliers
originates from the observation that the kernel of such operators
are of the form $K_0(|x-y|,|x+y|)$ and hence one can make use of
Funk-Hecke formula to study such operators. We will see that mixed
norm estimates for Hermite multipliers reduces to a vector valued
inequality for a sequence of operators, all of them related to the
original operator via Funk-Hecke formula. These component
operators can also be viewed as Laguerre multipliers in view of
the Hecke-Bochner formula.

Though in this paper we have only treated the Hermite operator,
all the results can be proved for the cases of Dunkl Harmonic
Oscillator and Special Hermite operator. The main tools used in
this paper viz., the Funk-Hecke formula and Hecke-Bochner identity
are also available for $h-$harmonics and bigraded spherical
harmonics. Hence the proofs of the main results can be suitably
modified to cover these more general operators. We plan to return
to some of these problems elsewhere.

The plan of the paper is as follows. In the next section we study
mixed norm estimates of integral operators with kernels of the
form $K_0(|x-y|,|x+y|).$ The results proved in Section 2 will be
applied in Section 3 to prove mixed norm estimates for the Hermite
semigroup and Bochner-Riesz means associated to Hermite
expansions. In Section 4 we study mixed norm estimates for
$g-$functions for the Hermite semigroup which will be used in
Section 5 to prove our main result on Hermite multipliers.

\section{Mixed norm estimates for some integral operators}
\setcounter{equation}{0}

In this section we study mixed norm estimates for certain
operators given by (singular) kernels having some special
properties. More precisely we consider operators of the form
$$Tf(x) = \int\limits_{\R^n} K(x,y) \ f(y) \ dy$$
where the kernel $K(x,y) = K_0(|x-y|,|x+y|)$ for some $K_0.$ We
are interested in estimates of the form
$$\|Tf\|_{L^{p,2}(\R^n)} \leq C \|f\|_{L^{p,2}(\R^n)}.$$ We prove that under
some mild assumptions on $T$ the above mixed norm estimates hold.
Examples of such operators are provided by the Hermite semigroup
$e^{-tH},$ Bochner-Riesz means $S_R^\delta$ associated to Hermite
expansions and more generally Hermite multipliers. We study these
operators in later sections using the main result proved in this
section.

We need several results from the theory of spherical harmonics on
$S^{n-1}.$ For each $m=0,1,2,\cdots,$ let $\mathcal{H}_m$ be the
space of spherical harmonics of degree $m.$ Let $Y_{m,j} , \
j=1,2,\cdots,d(m)$ be an orthonormal basis for $\mathcal{H}_m,$
where $d(m)$ is the dimension of $\mathcal{H}_m.$ We can take
these spherical harmonics to be real valued. It is well known that
the reproducing kernel
$$\sum\limits_{j=1}^{d(m)} \ Y_{m,j}(\omega) \ Y_{m,j}(\omega')$$
depends only on $\omega\cdot \omega'$ and hence there is a
function, denoted by $C_m^{\frac{n}{2}-1} (u),$ defined on
$[-1,1]$ such that
$$\sum\limits_{j=1}^{d(m)} \ Y_{m,j}(\omega) \ Y_{m,j}(\omega') = C_m^{\frac{n}{2}-1}(\omega\cdot\omega').$$
It can be shown that $C_m^{\frac{n}{2}-1}(u)$ are polynomials,
called ultra spherical polynomials of type $(\frac{n}{2}-1)$
satisfying the differential equation
$$(1-u^2)\  \varphi''(u)-(n-1)\ u\ \varphi'(u) + m(m+n-2)\ \varphi(u) =0.$$
Moreover, it is known that $$C_m^{\frac{n}{2}-1}(1)
=\sum\limits_{j=1}^{d(m)} \ Y_{m,j}(\omega) ^2 = d(m)
\omega_{n-1}^{-1},$$ $\omega _{n-1}$ being the surface measure of
$S^{n-1}.$ We let $$P_m^{\frac{n}{2}-1}\hspace{-0.1cm}(u) =
C_m^{\frac{n}{2}-1}(1)^{-1}\ C_m^{\frac{n}{2}-1}(u)$$ to stand for
the normalised ultra spherical polynomials. We note that
$|P_m^{\frac{n}{2}-1}\hspace{-0.1cm}(u)|\leq 1$ for all $u\in
[-1,1].$

For any function $F$ on $[-1,1]$ integrable with respect to the
measure $(1-u^2)^{\frac{n-3}{2}} du,$ the Funk-Hecke formula says
that
$$\int\limits_{S^{n-1}} F(\omega\cdot \omega') \ Y_{m,j}(\omega') \ d\omega' =
Y_{m,j}(\omega) \int\limits_{-1}^1 F(t) P_m^{\frac{n}{2}-1}\hspace{-0.1cm}(t) \
(1-t^2)^{\frac{n-3}{2}} \ dt.$$ We make use of this formula in the
proof of our main result in this section. Given the kernel $K(x,y)
= K_0(|x-y|,|x+y|)$ we define a sequence of kernels $K_m$ by
setting
$$K_m(x,y) = K_0(|x-y|,|x+y|) \ P_m^{\frac{n}{2}-1}\hspace{-0.1cm}(x'\cdot y') $$
where $x'=\frac{x}{|x|},\ y'=\frac{y}{|y|}.$ Consider the
operators
$$T_mf(x) = \int\limits_{\R^n} K_m(x,y) f(y) \ dy.$$
We observe that $T_mf$ is radial whenever $f$ is radial. Indeed,
when $f(y)=g(|y|)$ we have
$$T_mf(x) = \int\limits_0^\infty g(s) \left( \int \limits_{S^{n-1}} K_m(x,sy') dy'\right) s^{n-1}\ ds.$$
But then
$$\int \limits_{S^{n-1}}\ K_m(x,sy')\ dy' = \int \limits_{S^{n-1}} K_0(|x-sy'|,|x+sy'|)  \ P_m^{\frac{n}{2}-1}\hspace{-0.1cm}(x'\cdot y') \ dy'$$
is clearly radial in $x$ as the measure $dy'$ is rotation
invariant. Thus we can also consider $T_m$ as an operator on
$L^p(\R^+,r^{n-1}dr).$ We also let $\widetilde{T}$ stand for the
integral operator with kernel $|K(x,y)|.$

\begin{thm} \label{ttilde}
Let $T$ be an integral operator with kernel $K(x,y) =
K_0(|x-y|,|x+y|).$ If $\widetilde{T}$ is bounded on $L^p(\R^n)$
then $T$ satisfies the mixed norm estimate
$$\|Tf\|_{L^{p,2}(\R^n)} \leq C \|f\|_{L^{p,2}(\R^n)}.$$
\end{thm}

\begin{proof}
For a fixed $r,$ let us calculate the spherical harmonic
coefficients of $Tf(r\omega).$
$$\int\limits_{S^{n-1}} Tf(r\omega)\ Y_{m,j}(\omega)\ d\omega = \int\limits_{\R^n} \left(\ \int\limits_{S^{n-1}}
K(r\omega,y) Y_{m,j}(\omega) d\omega\right) f(y)\ dy.$$ If
$y=s\omega',$ the kernel, $K(x,y) =
K_0(|r\omega-s\omega'|,|r\omega+s\omega'|)$ is a function of
$\omega\cdot \omega'$ and hence by Funk-Hecke formula we have

\begin{eqnarray*}
&&\int\limits_{S^{n-1}}\hspace{-0.2cm} K(r\omega,s \omega') \ Y_{m,j}(\omega) \ d\omega\\
&&  = Y_{m,j}(\omega') \int\limits_{-1}^1
 K_0(q_+(r,s;u)^{\frac{1}{2}},q_-(r,s;u)^{\frac{1}{2}})\ P_m^{\frac{n}{2}-1}\hspace{-0.1cm}(u) \
(1-u^2)^{\frac{n-3}{2}} \ du\end{eqnarray*}
 where $q_{\pm}(r,s,u) = r^2+s^2\pm 2rsu.$
Therefore,
\begin{eqnarray*}
&&\int\limits_{S^{n-1}} Tf(r\omega)\ Y_{m,j}(\omega)\ d\omega\\
&&  = \int\limits_{0}^\infty f_{m,j}(s) \left(\int\limits_{-1}^1
K_0(q_+(r,s;u)^{\frac{1}{2}},q_-(r,s;u)^{\frac{1}{2}})  \
P_m^{\frac{n}{2}-1}\hspace{-0.1cm}(u) \ (1-u^2)^{\frac{n-3}{2}} \
du\right) s^{n-1} ds \end{eqnarray*}
 where $$f_{m,j}(s) = \int\limits_{S^{n-1}} f(s\omega')
 Y_{m,j}(\omega')\
d\omega'.$$ By the same Funk-Hecke formula applied to $K(x,y) \
P_m^{\frac{n}{2}-1}\hspace{-0.1cm}(x'\cdot y')$ we get
\begin{eqnarray*}
&&\int\limits_{S^{n-1}} K(x,y) \ P_m^{\frac{n}{2}-1}\hspace{-0.1cm}(x'\cdot y')\ dy'\\
&&  = \int\limits_{-1}^1
K_0(q_+(r,s;u)^{\frac{1}{2}},q_-(r,s;u)^{\frac{1}{2}})  \
P_m^{\frac{n}{2}-1}\hspace{-0.1cm}(u) \ (1-u^2)^{\frac{n-3}{2}} \
du. \end{eqnarray*} Consequently
\begin{eqnarray*}
   \int\limits_{S^{n-1}} Tf(r\omega)\ Y_{m,j}(\omega)\ d\omega &= & \int\limits_0^\infty \left( \ \int\limits_{S^{n-1}}
   f_{m,j}(s)\  K_m(x,y)\  dy'\right) s^{n-1}\ ds\\
   &=& T_m f_{m,j}(r)
\end{eqnarray*}
Thus we see that
$$Tf(r\omega) = \sum \limits_{m=0}^{\infty}\  \sum
\limits_{j=1}^{d(m)} \ T_m f_{m,j}(r) \ Y_{m,j}(\omega).$$ This
leads us to
$$\int\limits_{S^{n-1}} |Tf(r\omega)|^2\  d\omega = \sum \limits_{m=0}^{\infty} \ \sum
\limits_{j=1}^{d(m)}\  |T_mf_{m,j}(r)|^2.$$ By considering $T_m
f_{m,j}$ as $T_m$ applied to the radial function $f_{m,j}(|y|)$ we
see that
$$T_m f_{m,j}(x) = \int\limits_{\R^n} K(x,y) \ P_m^{\frac{n}{2}-1}\hspace{-0.1cm}(x'\cdot y') \ f_{m,j}(|y|) \ dy$$
and hence $|T_m f_{m,j}(x)| \leq \widetilde{T}(| f_{m,j}|)(x)$
since $P_m^{\frac{n}{2}-1}$ is bounded by $1.$ Thus the mixed norm
inequality will follow from
\begin{eqnarray*}
&&\hspace{-1cm}\int \limits_0^\infty \left( \sum
\limits_{m=0}^{\infty} \ \sum \limits_{j=1}^{d(m)}\
\widetilde{T}(|f_{m,j}|)(r)^2 \right)^\frac{p}{2} r^{n-1}
dr\\
&& \hspace{1cm} \leq C \int \limits_0^\infty \left( \sum
\limits_{m=0}^{\infty}\  \sum \limits_{j=1}^{d(m)}|f_{m,j}(r)|^2
\right)^\frac{p}{2} r^{n-1}  dr. \end{eqnarray*} If
$\widetilde{T}$ is bounded on $L^p(\R^n),$ the above is an
immediate consequence of a theorem of Marcinkiewicz and Zygmund
\cite{MZ}.
\end{proof}

In the next theorem we obtain a sufficient condition on the
operators $T_m$ so that the original operator satisfies a mixed
weighted norm inequality. On $\R^+$ consider the measure
$d\mu_\alpha = r^{2\alpha+1}dr$ with respect to which it becomes a
homogeneous space. We define Muckenhoupt's $A_p^\alpha$ weights
adapted to the measure $d\mu_ \alpha$ as follows. A positive
weight function $w$ is said to belong to $A_p^\alpha(\R^+)$ if it
satisfies
$$\left( \frac{1}{\mu_\alpha(Q)} \int\limits_Q w(r) \ d\mu_\alpha \right)
 \left(\frac{1}{\mu_\alpha(Q)} \int\limits_Q w(r)^{-\frac{1}{p-1}} \ d\mu_\alpha\right)^{p-1}\leq C$$
 for all intervals $Q\subset\R^+.$

Let $L^{p,2}(w)$ stand for the space of functions $f(r\omega)$ on
$\R^+\times S^{n-1}$ for which
$$\|f\|_{L^{p,2}(w)} = \left( \int
\limits_0^\infty \left(\ \int\limits_{S^{n-1}} |f(r\omega)|^2
d\omega\right)^\frac{p}{2}w(r) \  r^{n-1} \
dr\right)^{\frac{1}{p}}$$ are finite.

 \begin{thm}\label{tuniform}
 Let $T$ be as in the previous theorem and let $T_m$ be defined as
 before. If $T_m$ are uniformly bounded on $L^{p_0}(\R^+, w
 d\mu_{\frac{n}{2}-1})$ for some $p_0, \ 1 < p_0<\infty$ and all $w\in
 A_{p_0}^{\frac{n}{2}-1}\hspace{-0.1cm}(\R^+),$ then the mixed weighted norm
 inequality
 $$\|Tf\|_{L^{p,2}(w)} \leq C \|f\|_{L^{p,2}(w)}$$
holds for all $p,\ 1<p<\infty$ and all $w\in
A_{p}^{\frac{n}{2}-1}\hspace{-0.1cm}(\R^+).$
 \end{thm}

 As in the proof of Theorem \ref{ttilde} the mixed norm weighted inequality
 reduces to
 \begin{eqnarray*}
&&\hspace{-1cm}\int \limits_0^\infty \left( \sum
\limits_{m=0}^{\infty} \ \sum \limits_{j=1}^{d(m)}\ |T_m
f_{m,j}(r)|^2 \right)^\frac{p}{2}w(r) \ r^{n-1}\
 dr\\
&& \hspace{1cm} \leq C \int \limits_0^\infty \left( \sum
\limits_{m=0}^{\infty} \ \sum \limits_{j=1}^{d(m)}\ |f_{m,j}(r)|^2
\right)^\frac{p}{2} w(r) \ r^{n-1} \  dr. \end{eqnarray*} We can
easily deduce this inequality from the following extrapolation
theorem (Theorem 3.1 in \cite{Du1}, Proposition 3.3 in \cite{CR}).

\begin{thm} \label{xtrapoln}
Suppose that for some pair of nonnegative functions $f,g,$ for
some fixed $1<p_0<\infty$ and for all $w\in A_{p_0}^\alpha(\R^+)$
we have
$$\int\limits_0^\infty g(r) ^{p_0}\ w(r) \ d\mu_\alpha \leq C(w)\int\limits_0^\infty f(r) ^{p_0} w(r) \ d\mu_\alpha.$$
Then for all $1<q<\infty$ and all $w\in A_q^\alpha(\R^+)$ we have
$$\int\limits_0^\infty g(r) ^q w(r)\  d\mu_\alpha \leq C \int\limits_0^\infty f(r) ^q w(r)\  d\mu_\alpha.$$
\end{thm}
We can now easily deduce our theorem. The uniform boundedness of
$T_m$ on $L^{p_0}(\R^+,w \ d\mu_{\frac{n}{2}-1})$ can be used
along with the above theorem to give
$$\int\limits_0^\infty |T_m f(r)|^q w(r)\  d\mu_{\frac{n}{2}-1} \leq C(w)\int\limits_0^\infty |f(r)|^q w(r) \ d\mu_{\frac{n}{2}-1}$$
for any $1<q<\infty.$ If we take
$$g = \left( \sum
\limits_{m=0}^{\infty} \ \sum \limits_{j=1}^{d(m)}\ |T_m
f_{m,j}|^2 \right)^\frac{1}{2}, \ f=\left( \sum
\limits_{m=0}^{\infty}\  \sum \limits_{j=1}^{d(m)}\ |f_{m,j}|^2
\right)^\frac{1}{2},$$ then with $p=2$
$$\int\limits_0^\infty g(r) ^2 w(r) \ d\mu_{\frac{n}{2}-1} \leq C \int\limits_0^\infty f(r) ^2 w(r)\  d\mu_{\frac{n}{2}-1}$$
holds. By Theorem \ref{xtrapoln} we get
$$\int\limits_0^\infty g(r) ^p w(r) \ d\mu_{\frac{n}{2}-1} \leq C \int\limits_0^\infty f(r) ^p w(r)\  d\mu_{\frac{n}{2}-1}$$
for any $1<p<\infty, \ w\in A_p^{\frac{n}{2}-1}(\R^+).$ This
completes the proof.

\section{Hermite and Laguerre semigroups} \setcounter{equation}{0}

In this section we apply the results of the previous section to
study mixed weighted norm inequalities for the Hermite semigroup.
The Hermite semigroup $e^{-tH}, \ t>0,$ generated by the Hermite
operator $H=-\Delta +|x|^2$ on $\R^n$ is an integral operator
whose kernel $K_t(x,y)$ is explicitly known. Since the normalised
Hermite functions $\Phi_\alpha(x), \alpha \in \N^n$ are
eigenfunctions of $H$ with eigenvalues $(2|\alpha|+n)$ it follows
that
$$K_t(x,y) = \sum_{\alpha\in \N^n} e^{-(2|\alpha|+n)t} \ \Phi_\alpha(x) \Phi_\alpha(y).$$
In view of the Mehler's formula (Lemma 1.1.1, \cite{Th}) we see
that
$$K_t(x,y) = \pi^{-\frac{n}{2}} (\sinh 2t)^{-\frac{n}{2}} e^{-\frac{1}{2} (\coth 2t)(|x|^2+|y|^2)+(\csch 2t)x\cdot y}.$$
For each $t>0, \ e^{-tH}$ defines a bounded operator on
$L^p(\R^n), \ 1\leq ~p\leq ~\infty.$ The kernel of the operator
$e^{-tH}$ can be factorised as
$$K_t(x,y) = \pi^{-\frac{n}{2}} \ (\sinh 2t)^{-\frac{n}{2}} e^{-\frac{1}{4} (\coth t)|x-y|^2} e^{-\frac{1}{4} (\tanh t) |x+y|^2}$$
and hence we can make use of the results of the previous section
to prove mixed norm estimates. If $T = e^{-tH}$ then the operator
$T_m$ defined in the previous section with kernel $K_t(x,y) \
P_m^{\frac{n}{2}-1} (x'\cdot y')$ turns out to be a Laguerre
semigroup.

For each $\alpha>-1,$ let $L_k^\alpha(r), \ r>0$ be the Laguerre
polynomials of type $\alpha$ and define the normalised Laguerre
functions $\psi_k^\alpha(r)$ by
$$\psi_k^\alpha(r) = \left( \frac{2\Gamma(k+1)}{\Gamma(k+\alpha+1)}\right)^{\frac{1}{2}} \ L_k^\alpha(r^2) e^{-\frac{1}{2}r^2}.$$
Then $\psi_k^\alpha, \ k=0,1,2, \cdots$ forms an orthonormal basis
for $L^2(\R^+, d\mu_\alpha)$ where $d\mu_\alpha(r) = r^{2\alpha+1}
dr.$ Moreover, $\psi_k^\alpha$ are eigenfunctions of the Laguerre
operator
$$L_\alpha = - \frac{d^2}{dr^2}+r^2 -\frac{2\alpha+1}{r} \frac{d}{dr}$$
with eigenvalues $(4k+2\alpha+2).$ The Laguerre semigroup
generated by $L_\alpha$ is again an integral operator whose kernel
is given by
$$K_t^\alpha (r,s) = \sum \limits_{k=0}^{\infty} e^{-(4k+2\alpha+2)} \psi_k^\alpha(r) \ \psi_k^\alpha(s).$$
Using the generating function identity((1.1.47) in \cite{Th}) we
can easily see that
$$K_t^\alpha(r,s)= (\sinh 2t)^{-1} \ e^{-\frac{1}{2} (\coth 2t)(r^2+s^2)} \ (rs)^{-\alpha} \ I_\alpha(rs \csch 2t),$$
where $I_\alpha$ are the modified Bessel functions, $I_\alpha(z)=
e^{-i\frac{\pi}{2}\alpha} J_\alpha(iz).$ We note that for
$\alpha>-\frac{1}{2},$
$$J_\alpha(z) = \frac{(z/2)^\alpha}{\Gamma(\frac{1}{2}) \Gamma(\alpha+1)} \int\limits_{-1}^1 e^{izt} \ (1-t^2)^{\alpha-\frac{1}{2}} \ dt$$
and consequently,
$$I_\alpha(z) = \frac{(z/2)^\alpha}{\Gamma(\frac{1}{2}) \Gamma(\alpha+1)} \int\limits_{-1}^1 e^{zt} \ (1-t^2)^{\alpha-\frac{1}{2}} \ dt.$$
We will make use of these formulas in the sequel. We denote the
Laguerre semigroup $e^{-tL_\alpha}$ by $T_t^\alpha.$

The Hermite and Laguerre semigroups are intimately connected.
Consider $T_m$ applied to a radial function $f\in
L^p(\R^+,r^{n-1}dr).$ Then
$$T_mf(r) = \int\limits_0^\infty f(s) \left( \ \int\limits_{S^{n-1}} K_t(rx',sy') \ P_m^{\frac{n}{2}-1} \hspace{-0.1cm}(x'\cdot y') \ dy'\right)
s^{n-1}\ ds.$$ We claim that $T_m f(r) = C_n \ r^m\
T_t^{\frac{n}{2}+m-1} \widetilde{f}(r)$ where
$\widetilde{f}(r)=r^{-m} f(r).$ This follows immediately once we
show that
$$\int\limits_{S^{n-1}} K_t(rx',sy')\  P_m^{\frac{n}{2}-1} \hspace{-0.1cm}(x'\cdot y')\  dy' = C_n\  r^m s^m K_t^{\frac{n}{2}+m-1}(r,s)$$
By Funk-Hecke formula the above kernel is given by
$$\int\limits_{-1}^1 K_t(r,s;u) \ P_m^{\frac{n}{2}-1} \hspace{-0.1cm}(u) \ (1-u^2)^{\frac{n-3}{2}}\ du.$$
where $$K_t(r,s;u) = \pi^{-\frac{n}{2}} \ (\sinh
2t)^{-\frac{n}{2}} e^{-\frac{1}{4} (\coth t)q_-(r,s;u)}
e^{-\frac{1}{4} (\tanh t) q_+(r,s;u)}$$ Recalling the explicit
form of the kernel $K_t(x,y),$ all we need is the result from the
following lemma.

\begin{lem}
For any $z\in \C$
$$\int\limits_{-1}^1 e^{zu} P_m^{\frac{n}{2}-1}\hspace{-0.1cm}(u) \ (1-u^2)^{\frac{n-3}{2}} \ du =
\Gamma\left(\frac{1}{2}\right) \ \Gamma\left(\frac{n-1}{2}\right)
\ \left(\frac{z}{2}\right)^{-\frac{n}{2}+1} \ I_{\frac{n}{2}+m-1}
(z).$$
\end{lem}
\noindent\begin{proof} The function
$P_m^{\lambda}(t)=C_m^\lambda(t) C_m^\lambda(1)^{-1}$ is given by
the Rodrigues type formula
$$(1-t^2)^{\frac{n-3}{2}} \ P_m^{\frac{n}{2}-1} (t) = \frac{(-1)^m}{2^m (\frac{n-1}{2})_m} \ \left(\frac{d}{dt}\right)^m (1-t^2)^{\frac{n-3}{2}+m}.$$
In view of this formula
$$\int\limits_{-1}^{1} e^{zu} P_m^{\frac{n}{2}-1} (u) \ (1-u^2)^{\frac{n-3}{2}} \ du =
\frac{(-1)^m}{2^m (\frac{n-1}{2})_m} \int\limits_{-1}^{1} e^{zu}
\left(\frac{d}{du}\right)^m (1-u^2)^{\frac{n-3}{2}+m} \ du.$$
Integrating by parts we see that
\begin{eqnarray*}
&&\int\limits_{-1}^{1} e^{zu} \left(\frac{d}{du}\right)^m
(1-u^2)^{\frac{n-3}{2}+m} \ du\\
&& \ \ \ \ = (-1)^m z^m \int\limits_{-1}^{1}
e^{zu} (1-u^2)^{\frac{n-3}{2}+m} \ du\\
&& \ \ \ \ = (-1)^m \ 2^{m+\frac{n}{2}-1}\  z^{-\frac{n}{2}+1} \
\Gamma\left(\frac{1}{2}\right)\ \Gamma\left(\frac{n-1}{2}+m\right)
I_{\frac{n}{2}+m-1}(z).
\end{eqnarray*}
Since
$\frac{\Gamma(\frac{n-1}{2}+m)}{(\frac{n-1}{2})_m}=\Gamma(\frac{n-1}{2})$
we get the lemma.
\end{proof}
If we recall the definition of the kernels
$K_t^{\frac{n}{2}+m-1}(r,s)$ and combine it with the result of the
above lemma we immediately obtain the following:
\begin{prop}\label{Laguerre}
When $f(x) = f_{m,j}(r) Y_{m,j}(\omega),$ we have
$$e^{-tH} f(r\omega) = c_n \ r^m \ Y_{m,j}(\omega) \ T_t^{\frac{n}{2}+m-1} \widetilde{f}_{m,j}(r)$$
where $c_n = \Gamma(\frac{1}{2}) \Gamma(\frac{n-1}{2})
2^{\frac{n}{2}-1}$ and $\widetilde{f}_{m,j}(s) = s^{-m}
f_{m,j}(s).$
\end{prop}
\begin{rem}
The above proposition can also be proved directly using
Hecke-Bochner formula for the Hermite projections (see Theorem
3.4.1 in \cite{Th}). However, the above approach comes in handy
when we try to prove uniform estimates for the family of operators
taking $f_{m,j}$ into $r^m T_t^{\frac{n}{2}+m-1}
\widetilde{f}_{m,j}.$
\end{rem}

We are now ready to prove the following mixed weighted norm
inequality for the Hermite semigroup. Since
$P_m^{\frac{n}{2}-1}\hspace{-0.1cm}(u)$ is bounded it follows that
$|T_mf(r)| \leq T_0|f|(r)=T_t^{\frac{n}{2}-1}|f|(r).$ From the
work \cite{KT} we know that the Laguerre semigroup $T_t^\alpha$
satisfies the weighted norm inequality
$$\int\limits_0^\infty |T_t^\alpha f(r)|^p \ w(r)\ d\mu_\alpha \leq C \int\limits_0^\infty |f(r)|^p \ w(r)\  d\mu_\alpha$$
for all $w\in A_{p,loc}^\alpha(\R^+), \ 1\leq p <\infty.$ Here
$A_{p,loc}^\alpha(\R^+)$ is defined as the set of all positive
weight functions satisfying
$$\left( \frac{1}{\mu_\alpha(Q)} \int\limits_Q w(r)\  d\mu_\alpha \right)
 \left(\frac{1}{\mu_\alpha(Q)} \int\limits_Q w(r)^{-\frac{1}{p-1}} \ d\mu_\alpha\right)^{p-1}\leq C$$
 for all $Q \subset \R^+$ of length less than or equal to one.
 Combining this with the classical theorem of Marcinkiewicz and
 Zygmund \cite{MZ} mentioned already we get the following.

 \begin{thm}
 The mixed norm inequalities
 $$\|e^{-tH}f\|_{L^{p,2}(w)} \leq C \|f\|_{L^{p,2}(w)}$$
hold for all $w\in A_{p,loc}^{\frac{n}{2}-1}, \ 1\leq p<\infty.$
 \end{thm}
Indeed, let $R$ be the radialisation operator defined by
$$Rf(r) = \int\limits_{S^{n-1}} f(r\omega)\  d\omega.$$
Defining $Sf(x) = T_t^{\frac{n}{2}-1} (Rf \cdot w^{-1/p})(x) \
w(|x|)^{\frac{1}{p}}$ we see that
\begin{eqnarray*}
\int \limits_{\R^n} |Sf(x)|^p \ dx &=& C \int\limits_0^\infty
|T_t^{\frac{n}{2}-1}(Rf
\cdot w^{-1/p})(r)|^p\  w(r)\  d\mu_{\frac{n}{2}-1}\\
&\leq& C \int\limits_0^\infty |Rf (r)|^p \ r^{n-1} \ dr.
\end{eqnarray*}
Since
$$|Rf(r)|^p \leq \left(\ \int\limits_{S^{n-1}} d\omega\right)^{\frac{p}{p'}} \int\limits_{S^{n-1}} |f(r\omega)|^p\  d\omega$$
it follows that
$$\int\limits_{\R^n} |Sf(x)|^p dx \leq C \int\limits_{\R^n} |f(x)|^p
dx.$$ By the theorem of Marcinkiewicz and Zygmund we have
$$\int\limits_{\R^n} (\sum\limits_{j=0}^\infty \ |Sf_j(x)|^2)^\frac{p}{2} \ dx \leq
 C \int\limits_{\R^n} (\sum\limits_{j=0}^\infty\  |f_j(x)|^2)^\frac{p}{2}\ dx$$
 for any sequence $f_j \in L^p(\R^n).$ Applying this to the radial
 functions $g_{m,j}(x) = f_{m,j}(|x|) w(|x|)^\frac{1}{p}$ where
 $f_{m,j} \in L^p(\R^+,w(r)\ d\mu_{\frac{n}{2}-1})$ we obtain
\begin{eqnarray*}
&&\hspace{-1cm}\int \limits_0^\infty \left(\sum
\limits_{m=0}^{\infty} \ \sum \limits_{j=1}^{d(m)}\
|T_t^{\frac{n}{2}-1} f_{m,j}(r)|^2 \right)^\frac{p}{2}
w(r)\  d\mu_{\frac{n}{2}-1} \\
&&\hspace{1cm}\leq C\int \limits_0^\infty \left( \sum
\limits_{m=0}^{\infty}\  \sum \limits_{j=1}^{d(m)}\ |f_{m,j}(r)|^2
\right)^\frac{p}{2} w(r)\  d\mu_{\frac{n}{2}-1}.
\end{eqnarray*} This proves the theorem.

\begin{rem}
In view of Proposition \ref{Laguerre} the main result can be
restated as the following vector valued inequality for the
Laguerre semigroups $T_t^{\frac{n}{2}+m-1}:$ for $1\leq p\leq
\infty, w\in A_{p,loc}^{\frac{n}{2}-1}$ we have
\begin{eqnarray*}
&&\hspace{-1cm}\int \limits_0^\infty \left(\sum
\limits_{m=0}^{\infty} \ \sum \limits_{j=1}^{d(m)}\  r^{2m}\
|T_t^{\frac{n}{2}+m-1} \widetilde{f}_{m,j}(r)|^2
\right)^\frac{p}{2} w(r)\
d\mu_{\frac{n}{2}-1}\\
 &&\hspace{1cm}\leq C_{n,t} \int \limits_0^\infty \left( \sum
\limits_{m=0}^{\infty} \ \sum \limits_{j=1}^{d(m)}\ |f_{m,j}(r)|^2
\right)^\frac{p}{2} w(r) \ d\mu_{\frac{n}{2}-1}.
\end{eqnarray*}
\end{rem}

Using Theorem 2.1 we can also prove mixed norm estimates for Riesz
means associated to Hermite expansions on $\R^n, \ n\geq 2.$ For
$\delta\geq 0$ we define the Riesz means of order $\delta$ by
$$S_R^\delta f(x) = \sum \left( 1-\frac{2k+n}{R}\right)_+^\delta \ P_kf(x)
$$
where $P_k$ are the Hermite projection operators defined by
$$P_k f(x) = \int\limits_{\R^n} \Phi_k(x,y) \ f(y) \ dy$$
with $\Phi_k(x,y) = \sum\limits_{|\alpha| =k} \Phi_\alpha(x)
\Phi_\alpha(y), \ \Phi_\alpha$ being the normalised Hermite
functions on $\R^n.$ The basic result for the Riesz means is the
following: if $\delta>\frac{n-1}{2},$ the operator $S_R^\delta$
are all uniformly bounded on $L^p(\R^n),\\ 1 \leq p\leq
\infty.$(Theorem 3.3.2 in \cite{Th}). This follows from the fact
that (see Theorem 3.3.1 \cite{Th})
$$\sup \limits_{x\in \R^n} \int\limits_{\R^n} |s_R^\delta(x,y)| \ dy\leq C$$
where $C$ is independent of $R.$ Here $s_R^\delta(x,y)$ is the
kernel associated to $S_R^\delta,$ viz.,
$$s_R^\delta(x,y) = \sum\limits_{k}  \left( 1-\frac{2k+n}{R}\right)_+^\delta \ \Phi_k(x,y).$$
It follows that the operator $\widetilde{S}_R^\delta$ given by
$$\widetilde{S}_R^\delta f(x) = \int\limits_{\R^n}
|s_R^\delta(x,y)| \ f(y)\ dy$$ is also bounded on $L^p(\R^n), \
1\leq p\leq \infty$ for $\delta>\frac{n-1}{2}.$ Using this fact we
can easily obtain the following mixed norm estimates for
$S_R^\delta.$

\begin{thm}
Let $n\geq 2$ and $\delta >\frac{n-1}{2}.$ Then we have the
uniform estimates
$$\|S_R^\delta f\|_{L^{p,2}(\R^n)} \leq C \|f\|_{L^{p,2}(\R^n)}.$$
 for all $f\in L^p(\R^n),
1\leq p\leq \infty.$ (Here $C$ is independent of $R).$
\end{thm}
\begin{proof}
To prove this theorem, all we have to do is to observe that
$s_R^\delta(x,y) = K_0(|x-y|,|x+y|).$ To see this it is enough to
check that $\Phi_k(x,y)$ have the same property. The heat kernel
$K_t(x,y)$ can be defined even for complex $t$ from the open unit
disc and we have
$$t^k \ \Phi_k(x,y) = \frac{1}{2\pi} \int\limits_0^{2\pi} K_{te^{i\theta}} (x,y)\  e^{-ik\theta} \ d\theta.$$
Since $K_{te^{i\theta}} (x,y)$ depends only on $x\cdot y$ the same
is true for $\Phi_k(x,y).$ And this proves the theorem.
\end{proof}
\begin{rem}
The above theorem does not say anything for $S_R^\delta$ when
$\delta$ is below the critical index $\frac{n-1}{2}.$ We
conjecture that for $0\leq \delta\leq \frac{n-1}{2},$ the mixed
norm estimates of the theorem holds for all $p$ satisfying
$\frac{2n}{n+1+2\delta}<p<\frac{2n}{n-1-2\delta}.$ This is weaker
than the Bochner-Riesz conjecture for Hermite expansions and hence
stands a better chance of getting proved.
\end{rem}

\section{$g-$functions associated to the Hermite semigroups} \setcounter{equation}{0}

For each positive integer $k,$ the $g-$function $g_k$ associated
to the Hermite semigroup $T_t=e^{-tH}$ is defined by
$$g_k(f,x)^2 = \int\limits_0^\infty |\partial _t^k T_t f(x)|^2 \ t^{2k-1} \ dt.$$
These $g-$functions have been studied in \cite{Th} and the
following estimates are known: for $1<p<\infty,$
$$C_1 \|f\|_p \leq \|g_k(f)\|_p\leq C_2 \|f\|_p$$
and when $p=2$ we have
$$\|g_k(f)\|_2 = C_k \|f\|_2.$$
In this section we are interested in obtaining mixed norm
estimates for $g_k-$functions. The $g_k-$functions are singular
integral operators with kernels taking values in the Hilbert space
$L^2(\R^+,t^{2k-1}dt).$ More precisely the kernel is given by
$\partial_t^k K_t(x,y)$ and hence the results in Section 1 can be
applied to study these functions.

Let us consider the $L^2-$norm of $g_k(f,x)$ in the angular
variable first.
$$\int\limits_{S^{n-1}} g_k(f,r\omega)^2 \ d\omega = \int\limits_{S^{n-1}}
 \left(\int\limits_0^\infty t^{2k-1} \ |\partial_tT_t f(r\omega)|^2 \ dt\right) d\omega.$$
Interchanging the order of integration we see that
$$\int\limits_{S^{n-1}} |\partial_t^k T_t f(r\omega)|^2 \ d\omega=\sum
\limits_{m=0}^{\infty} \ \sum \limits_{j=1}^{d(m)}\
|\int\limits_{S^{n-1}} \partial_t^k T_t f(r\omega) \
Y_{m,j}(\omega) \ d\omega|^2.$$ It follows that
$$\int\limits_{S^{n-1}} \partial_t^k T_t f(r\omega) \ Y_{m,j}(\omega)
\ d\omega = \partial_t^k T_{t,m} f_{m,j}(r)$$ where the kernel of
$T_{t,m}$ is $K_t(x,y) P_m^{\frac{n}{2}-1}(x'\cdot y').$ We also
know that
$$T_{t,m}f_{m,j}(r) =r^m \ T_t^{\frac{n}{2}+m-1} \widetilde{f}_{m,j}(r)$$
and consequently,
\begin{eqnarray*}
&&\hspace{-1cm}\int\limits_{S^{n-1}} \left(\int \limits_0^\infty
t^{2k-1} \ |\partial_t^k T_t f(r\omega)|^2 \ dt\right) d\omega\\
 &&\hspace{1cm}= \sum
\limits_{m=0}^{\infty} \ \sum \limits_{j=1}^{d(m)}\ \left( \int
\limits_0^\infty  t^{2k-1} \ |\partial_t^k T_t^{\frac{n}{2}+m-1}
\widetilde{f}_{m,j}(r)|^2 dt \right)r^{2m}.
\end{eqnarray*}
Thus we are led to study $g_k-$functions defined for the Laguerre
semigroups.

For each $m=0,1,2,\cdots,$ let us define
$$g_{k,m}(f,x)^2 = \int\limits_0^\infty |\partial _t^k T_{t,m} f(x)|^2 \ t^{2k-1} \ dt,$$
where $f$ is a radial function on $\R^n.$ The $L^2-$theory of
$g_k$ functions for Laguerre semigroups immediately gives us
\begin{eqnarray*}
&&\hspace{-1cm}\int \limits_0^\infty \left( \sum
\limits_{m=0}^{\infty} \ \sum \limits_{j=1}^{d(m)}\
g_{k,m}(f_{m,j},r)^2
\right)^\frac{p}{2} r^{n-1} \ dr\\
 &&\hspace{1cm}= C_{k,n}  \int
\limits_0^\infty \left(\sum \limits_{m=0}^{\infty} \ \sum
\limits_{j=1}^{d(m)}\  |f_{m,j}(r)|^2 \right)^\frac{p}{2} r^{n-1}
\ dr
\end{eqnarray*}
for $p=2.$ We are interested in proving the inequality for general
$p, \ 1<p<\infty.$ By showing that $g_{k,m}$ are singular integral
operators uniformly bounded on $L^p(\R^+,w d\mu_{\frac{n}{2}-1})$
we will obtain the following result.

\begin{thm}\label{gkbdd}
For every $k=1,2,\cdots,$ the $g_k-$functions satisfy
$$\|g_k(f)\|_{L^{p,2}(w)} \leq  C_k \|f\|_{L^{p,2}(w)}$$
for all $w\in A_p^{\frac{n}{2}-1} (\R^+), \ 1<p<\infty.$
\end{thm}
In order to prove this theorem we need to recall
Calder\'{o}n-Zygmund theory of singular integral operators on the
homogeneous space $(\R^+,d\mu_\alpha)$ developed by Calder\'{o}n
\cite{Ca}. Let $B(a,b)$ stand for the ball of radius $b>0$ centred
at $a\in \R^+$ and let $\mu_\alpha (B(a,b))$ stand for its
measure. Note that $g_{k,m}$ can be considered as a (singular)
integral operator with $L^2(\R^+, t^{2k-1} dt) -$valued kernel
given by
$$\partial_t^k K_m(r,s;t) = \int\limits_{S^{n-1}}\partial_t^k K_t(x,y) \ P_m ^{\frac{n}{2}-1}(x'\cdot y') \ dy',$$
$x=rx',\ y=sy'.$ It is therefore, enough to show that $K_m$
satisfy the following Calder\'{o}n-Zygmund estimates uniformly in
$m.$

\begin{prop} \label{kernelest}
The kernels $K_m(r,s;t)$ satisfy the following uniform estimates:
for $r\neq s$
\renewcommand{\descriptionlabel}[1]{\hspace{0cm}\textit{#1}}
\begin{description}
    \item[(i)] $\left( \int\limits_0^\infty | \partial_t^k K_m(r,s;t)|^2 \ t^{2k-1} \ dt
    \right)^{\frac{1}{2}} \leq \frac{C_1}{\mu_{n/2-1}
    (B(r,|r-s|))}$\\
    \item[(ii)] $\left( \int\limits_0^\infty | \partial_t^k \partial_r K_m(r,s;t)|^2 \ t^{2k-1} \ dt
    \right)^{\frac{1}{2}} \leq \frac{C_2}{|r-s| \ \mu_{n/2-1}
    (B(r,|r-s|))}$
\end{description}
where $C_1$ and $C_2$ are independent of $m.$
\end{prop}
The measure $\mu_{\frac{n}{2}-1}
    (B(r,|r-s|))$ has been estimated in \cite{NS}(See Proposition
    3.2)
    according to which
    $$\mu_{\frac{n}{2}-1}(B(r,|r-s|)) \approx |r-s| (r+s)^{n-1}.$$
    We will also make use of the following Lemma due to Nowak and Stempak \cite{NS} (see Lemma
    5.3 in \cite{CR2}).

    \begin{lem} \label{ABC}
    Let $c \geq \frac{1}{2}, \ 0<B<A$ and $\lambda>0.$ Then
    $$\int\limits_0^1 \frac{(1-u)^{c-\frac{1}{2}}}{(A-Bu)^{c+\lambda+\frac{1}{2}}} \ du \leq \frac{C}{A^{c+\frac{1}{2}} (A-B)^{\lambda}}.$$
    \end{lem}
    We also need the following estimates on the kernel
    $\partial_t^k K_t(x,y)$ proved in \cite{Th} (see inequalities (4.1.12) and
    (4.1.13)) and \cite{ST}.

    \begin{prop} \label{kernelbk}
    We have the following estimates (for $x\neq y):$
    \renewcommand{\descriptionlabel}[1]{\hspace{0cm}\textit{#1}}
    \begin{description}
        \item[(i)] $\left( \int\limits_0^\infty |\partial_t^k K_t(x,y)|^2 \ t^{2k-1} \ dt \right)^{\frac{1}{2}} \leq
        C_k|x-y|^{-n}$
        \item[(ii)] $\left( \int\limits_0^\infty |\nabla_x \partial_t^k K_t(x,y)|^2 \ t^{2k-1} \ dt \right)^{\frac{1}{2}} \leq
        C_k|x-y|^{-n-1}$
    \end{description}
    \end{prop}
    We can now easily prove the estimates in Proposition \ref{kernelest}. In
    view of the above estimate (i) in Proposition \ref{kernelbk} and the fact
    that $P_m^{\frac{n}{2}-1}(u)$ is bounded we see that
    \begin{eqnarray*}
       && \hspace{-1cm}\left( \int\limits_0^\infty |\partial_t^k K_m(r,s;t)|^2 \ t^{2k-1} \ dt \right)^{\frac{1}{2}}  \\
       && \hspace{1cm} \leq \int\limits_{S^{n-1}} \left( \int\limits_0^\infty |\partial_t^k K_t(rx',sy')|^2 \ t^{2k-1} \ dt \right)^{\frac{1}{2}} dy'\\
       &&\hspace{1cm} \leq C \int\limits_{S^{n-1}}
       |rx'-sy'|^{-n}
       \ dy'.
    \end{eqnarray*}
    The last integral is equal to a constant multiple of
    $$\int\limits_0^\pi (r^2+s^2-2rs\cos \theta)^{-\frac{n}{2}} \ (\sin \theta)^{n-2} \ d\theta$$
    which is estimated by
    $$\int\limits_0^1 (r^2+s^2-2rsu)^{-\frac{n}{2}} \ (1-u)^{\frac{n-3}{2}} \ du.$$
    By appealing to Lemma \ref{ABC} with $A=r^2+s^2, \ B=2rs, \
    c=\frac{n}{2}-1, \ \lambda=\frac{1}{2}$ we obtain the estimate
    $$\left( \int\limits_0^\infty |\partial_t^k K_m(r,s;t)|^2 \ t^{2k-1} \ dt \right)^{\frac{1}{2}} \leq C|r-s|^{-1} \ (r^2+s^2)^{-\frac{n-1}{2}}.$$
    The desired estimate follows as $|r-s|\
    (r^2+s^2)^{\frac{n-1}{2}}$ is comparable to
    $\mu_{\frac{n}{2}-1}(B(r,|r-s|)).$

    In order to get the estimate on the derivative we note that
    $$\partial_r\partial_t^k K_m(r,s;t) = \int\limits_{S^{n-1}} \partial_r\partial_t^k K_t(x,y) \ P_m^{\frac{n}{2}-1}(x'\cdot y') \ dy'.$$
    Since $\partial_r\partial_t^k K_t(x,y)=\sum\limits_{j=1}^{n} \
    \frac{\partial}{\partial x_j}\partial_t^k K_t(x,y)
    x_j'$ we can estimate
    $$\left( \int\limits_0^\infty |\partial_r \partial_t^k K_m(r,s;t)|^2 \ t^{2k-1} \ dt \right)^{\frac{1}{2}}$$
    in terms of
    $$\int\limits_{S^{n-1}} \left( \int\limits_0^\infty | \nabla_x \partial_t^k K_t(x,y)|^2 \ t^{2k-1} \ dt \right)^{\frac{1}{2}} dy'.$$
    This, in view of (ii) of Proposition \ref{kernelbk} leads to the integral
    $$\int\limits_{S^{n-1}}  |rx'-sy'|^{-n-1} \ dy'$$
    which can be estimated, as above, leading to $|r-s|^{-2}
    (r^2+s^2)^{-\frac{(n-1)}{2}}.$ This proves the second estimate
    in Proposition \ref{kernelest}.

    Thus the $g_k-$functions $g_{k,m}$ are all uniformly bounded
    on the space
    $L^p(\R^+,wd\mu_{\frac{n}{2}-1})$ for any $1<p<\infty, \ w\in
    A_p^{\frac{n}{2}-1}(\R^+)$ and consequently, the $g_k-$functions
    satisfy mixed weighted norm inequalities. This proves
    Theorem \ref{gkbdd}.

    Polarising the identity $\|g_k(f)\|_2 = C_k \|f\|_2$ and using
    the boundedness of $g_k-$functions we can prove the reverse
    inequality
$$C_1 \|f\|_{L^{p,2}(\R^n)} \leq \|g_k(f)\|_{L^{p,2}(\R^n)}.$$ Indeed, polarising
$\|g_k(f)\|_2 = C_k \|f\|_2$ and performing the integration over
$S^{n-1}$ we get
\begin{eqnarray*}
&&\hspace{-1cm} C_k \int \limits_0^\infty \left(\sum
\limits_{m=0}^{\infty} \ \sum \limits_{j=1}^{d(m)}\ f_{m,j}(r)
\overline{h_{m,j}}(r) \right)r^{n-1} \ dr \\
&&\hspace{1cm}=  \int \limits_0^\infty \int \limits_0^\infty
 \sum \limits_{m=0}^{\infty}\  \sum \limits_{j=1}^{d(m)}\
 \partial_t^k T_{t,m}f_{m,j}(r) \ \overline{\partial_t^k T_{t,m}h_{m,j}(r)} \ t^{2k-1} \ dt \ r^{n-1} \ dr.
\end{eqnarray*}
The right hand side is dominated by
$$\int \limits_0^\infty
\left( \sum \limits_{m=0}^{\infty}\  \sum \limits_{j=1}^{d(m)}\
g_{k,m}(f_{m,j},r) \  g_{k,m}(h_{m,j},r)\right) \ r^{n-1} \ dr.$$
 Applying H\"{o}lder's inequality for the vector valued functions
 $(g_{k,m}(f_{m,j}))$ and $(g_{k,m}(h_{m,j})),$ and using the
 boundedness of $g_{k,m}-$functions the above is dominated by
 $\|f\|_{L^{p,2}(\R^n)}\ \|h\|_{L^{p',2}(\R^n)}.$ By
taking supremum over all $h\in L^{p',2}(\R^n)$ we get the required
inequality.

\section{Mixed norm estimates for Hermite multipliers} \setcounter{equation}{0}

In this section we will prove mixed norm estimates for Hermite
multipliers making use of our results on $g-$functions proved in
Section 4. Given a bounded function $\varphi$ defined on the set
of all positive integers we can define a bounded linear operator
on $L^2(\R^n)$ by means of spectral theorem:
$$\varphi(H)f = \sum\limits_{k=0}^{\infty} \ \varphi(2k+n) \ P_kf.$$
This is clearly a bounded operator on $L^2(\R^n)$ but without
further assumptions on $\varphi$ it may not be possible to extend
$\varphi(H)$ to $L^p(\R^n), \ p\neq 2$ as a bounded linear
operator. Consider the finite difference operator $\Delta$ defined
by
$$\Delta \varphi(k) = \varphi(k+1)-\varphi(k)$$
and define $\Delta^{j+1}\varphi(k) = \Delta(\Delta^j \varphi)(k)$
for $j=1,2,\cdots.$ The following theorem has been proved in
\cite{Th}, see Theorem 4.2.1.

\begin{thm}\label{mtrbook}
Assume that $k>\frac{n}{2}$ is an integer and the function
$\varphi$ satisfies $|\Delta^j \varphi(k)| \leq C_j (2k+n)^{-j}, \
j=0,1,2,\cdots ,k.$ Then $\varphi(H)$ is bounded on $L^p(\R^n)$
for all $1<p<\infty.$
\end{thm}

Actually the theorem is true for more general multipliers but we
have stated it in the above form for the sake of simplicity. The
proof relies on the $g-$function estimates, viz.,
$$C_1\|f\|_p \leq \|g(f)\|_p \leq C_2 \|f\|_p$$
for the $g-$function defined for the semigroup $e^{-tH}.$ Another
ingredient is the boundedness of $g_k^*$ functions: when $k>
\frac{n}{2}$ and $p>2$ we have
$$\|g_k^*(f)\|_p \leq C\|f\|_p.$$
Here $g_k^*(f,x)$ is defined by
$$g_k^*(f,x) ^2 = \int\limits_0^\infty \int\limits_{\R^n} t^{-n/2} \ (1+t^{-1}|x-y|^2)^{-k} \ |\partial_t T_t f(y)|^2 \ dy \ t \ dt.$$
Under the hypothesis of Theorem \ref{mtrbook} one proves that
$$g_{k+1}(\varphi(H)f,x) \leq C\ g_k^*(f,x)$$
and this can be used in conjunction with the boundedness of $g_k$
and $g_k^*$ functions to prove the multiplier theorem.

For each $m,$ we introduce the following $g_{k,m}^*$ functions:
$$g_{k,m}^*(f,x)^2=\int\limits_0^\infty \int \limits_{\R^n} t^{-\frac{n}{2}} \ (1+t^{-1}|x-y|^2)^{-k} \
 |\partial_t T_{t,m}f(y)|^2 \ dy \ t \ dt$$
 where $f$ is a radial function on $\R^n.$ It is then clear that
 $g_{k,m}^*(f,x)$ is a radial function of $x$ and hence we consider
 $g_{k,m}^*(f)$ as defined on $L^p(\R^+,r^{n-1}dr).$ For any radial
 function $h$ on $\R^n$ look at
\begin{eqnarray*}
&&\hspace{-1cm} \int \limits_{\R^n} g_{k,m}^*(f,x)^2 \ h(x) \ dx\\
&&\hspace{1cm}=  \int \limits_0^\infty  |\partial_t T_{t,m}f(y)|^2
\left( \ \int \limits_{\R^n} t^{-\frac{n}{2}} \
(1+t^{-1}|x-y|^2)^{-k} \ h(x) \ dx\right) t \ dt.
\end{eqnarray*}
As $h$ is radial, the inner integral is given by
$$C_n \int\limits_0^\infty h(r) \left( \ \int \limits_{S^{n-1}} t^{-\frac{n}{2}} \
(1+t^{-1}|rx'-sy'|^2)^{-k} dy'\right) r^{n-1} \ dr$$ where $h(r) =
h(x)$ with $|x|=r.$ For $k>\frac{n}{2}$ the function
$t^{-\frac{n}{2}} \ (1+t^{-1}r^2)^{-k}$ is integrable over $\R^+$
with respect to $r^{n-1} \ dr$ and the integral
$$\int\limits_{S^{n-1}} t^{-\frac{n}{2}} \
(1+t^{-1}|rx'-sy'|^2)^{-k} \ dy'$$ is nothing but the generalised
Euclidean translation of $t^{-\frac{n}{2}} \ (1+t^{-1}r^2)^{-k},$
see Stempak \cite{St}, and \cite{Th}. Consequently the integral
$$\int \limits_{\R^n} t^{-\frac{n}{2}} \ (1+t^{-1}|x-y|^2)^{-k} \ h(x) \ dx$$
is dominated by the maximal function $M_{\frac{n}{2}-1}h(s).$ Thus
we have obtained
\begin{eqnarray}\label{maxl}
&&\hspace{-1cm}  \int \limits_0^\infty g_{k,m}^*(f,r)^2 \ h(r) \ r^{n-1} \ dr\\
&&\hspace{1cm}\leq  C \int \limits_0^\infty  g_{k,m}(f,r)^2 \
M_{\frac{n}{2}-1}h(r) \ r^{n-1} \ dr.\nonumber
\end{eqnarray}
Here the constant $C$ is independent of $m.$ By taking $h=1,$ the
boundedness of $g_{k,m}$ on $L^2(\R^+,d\mu_{\frac{n}{2}-1})$ leads
to the same for $g_{k,m}^*.$ By standard arguments one can prove
the uniform estimates
$$
\int \limits_0^\infty (g_{k,m}^*(f,r))^p \ r^{n-1} \ dr\leq C\int
\limits_0^\infty |f(r)|^p \ r^{n-1} \ dr,
$$
for all $p\geq 2.$ We will make use of these estimates in the
following theorem.

\begin{thm}
Let $\varphi$ be as in Theorem \ref{mtrbook}. Then $\varphi(H)$
satisfies the mixed norm estimates
$$\|\varphi(H)f\|_{L^{p,2}(\R^n)} \leq C \|f\|_{L^{p,2}(\R^n)}$$
for $1<p<\infty.$
\end{thm}
In order to prove this theorem we proceed as follows. Let
$\varphi$ be as in the theorem and let $F=\varphi(H)f.$ As in the
proof of Theorem 4.2.1 in \cite{Th} we have the estimate
$g_{k+1}(F,x) \leq C_k \ g_k^*(f,x)$ provided $k>\frac{n}{2}.$ The
reverse inequality $\|F\|_{L^{p,2}(\R^n)} \leq C
\|g_{k+1}(F)\|_{L^{p,2}(\R^n)}$ together with the above estimate
gives us
$$\|F\|_{L^{p,2}(\R^n)} \leq C_k \ \|g_k^*(f)\|_{L^{p,2}(\R^n)}$$
and we will show that $\|g_k^*(f)\|_{L^{p,2}(\R^n)} \leq C
\|f\|_{L^{p,2}(\R^n)}$ for $p\geq 2$ which will prove the theorem
for $p\geq 2.$ By duality we can take care of the case $1<p<2.$
Note that
$$\int\limits_{S^{n-1}} g_k^*(f,rx')^2 \ dx' = \sum \limits_{m=0}^{\infty}\  \sum \limits_{j=1}^{d(m)}\ g_{k,m}^*(f_{m,j},r)^2$$
which follows from the fact that
$$\int \limits_{S^{n-1}} t^{-\frac{n}{2}} \
(1+t^{-1}|x-y|^2)^{-k} \ dx'$$ is a radial function of $y$ and
that
$$\int \limits_{S^{n-1}} |\partial_t T_t f(y)|^2 \ dy' = \sum \limits_{m=0}^{\infty}\  \sum \limits_{j=1}^{d(m)}\ |\partial_t T_{t,m} f_{m,j}(s)|^2$$
as observed earlier in the previous section.

Therefore, we are left with proving the inequality
\begin{eqnarray*}
&&\hspace{-1cm} \int \limits_0^\infty \left(\sum \limits_{m=0}^{\infty}\  \sum \limits_{j=1}^{d(m)}\
g_{k,m}^*(f_{m,j},r)^2\right)^{\frac{p}{2}} r^{n-1} \ dr\\
&&\hspace{1cm}\leq C  \int \limits_0^\infty \left(\sum
\limits_{m=0}^{\infty}\  \sum \limits_{j=1}^{d(m)}\
|f_{m,j}(r)|^2\right)^{\frac{p}{2}} r^{n-1} \ dr.
\end{eqnarray*}
The following argument is essentially taken from Rubio de Francia
\cite{Ru}, see the proof of the main theorem. For $p>2$ let
$q=\frac{p}{2}$ and take $h\in L^{q'}$ with $\|h\|_{q'}=1$ and
\begin{eqnarray*}
&&\hspace{-1cm} \int \limits_0^\infty \left(\sum
\limits_{m=0}^{\infty}\  \sum \limits_{j=1}^{d(m)}\
g_{k,m}^*(f_{m,j},r)^2\right)^{\frac{p}{2}} r^{n-1} \ dr\\
&&\hspace{1cm}=  \left(\int \limits_0^\infty \left(\sum
\limits_{m=0}^{\infty}\  \sum \limits_{j=1}^{d(m)}\
g_{k,m}^*(f_{m,j},r)^2\right)h(r) \  r^{n-1} \ dr\right)^q.
\end{eqnarray*}
The uniform inequality (\ref{maxl}) gives us
\begin{eqnarray*}
&&\hspace{-1cm} \int \limits_0^\infty \left(\sum
\limits_{m=0}^{\infty}\  \sum \limits_{j=1}^{d(m)}\
g_{k,m}^*(f_{m,j},r)^2\right)h(r) \ r^{n-1} \ dr\\
&&\hspace{1cm}\leq C   \int \limits_0^\infty \sum
\limits_{m=0}^{\infty}\  \sum \limits_{j=1}^{d(m)}\
g_{k,m}(f_{m,j},r)^2 \ v(r) \ r^{n-1} \ dr
\end{eqnarray*}
where $v(r) = M_{\frac{n}{2}-1} h(r) \in L^{q'}(\R^+,r^{n-1}dr).$
As in \cite{Ru} the last integral is dominated by
$$\sum
\limits_{m=0}^{\infty}\  \sum \limits_{j=1}^{d(m)} \int
\limits_0^\infty g_{k,m}(f_{m,j},r)^2 \ u(r) \ r^{n-1} \ dr$$
where $u(r) = (M_{\frac{n}{2}-1} v^s)^{\frac{1}{s}} \in
L^{q'}(\R^+,r^{n-1} \ dr)$ provided that $1<s<q'.$ Now it can be
shown that the function $u \in A_1^{\frac{n}{2}-1}(\R^+).$ Such a
result for $\R^n$ with Lebesgue measure, due to Coifman, has been
proved in Theorem 7.7 \cite{Du2}. The same proof works for $\R^+$
with the measure $\mu_{\frac{n}{2}-1}.$ Since $u \in
A_1^{\frac{n}{2}-1}(\R^+) \subset A_2^{\frac{n}{2}-1}(\R^+),$ the
weighted norm inequality for $g_{k,m}$ gives
\begin{eqnarray*}
  &&\hspace{-1cm} \sum
\limits_{m=0}^{\infty}\  \sum \limits_{j=1}^{d(m)} \int
\limits_0^\infty g_{k,m}(f_{m,j},r)^2 \ u(r) \ r^{n-1} \ dr \\
   &&\hspace{1cm} \leq C \int
\limits_0^\infty \left( \sum
\limits_{m=0}^{\infty}\  \sum \limits_{j=1}^{d(m)} \ |f_{m,j}(r)|^2 \right)  \ u(r) \ r^{n-1} \ dr\\
   &&\hspace{1cm} \leq C \|u\|_{q'} \ \left(\int
\limits_0^\infty \left( \sum \limits_{m=0}^{\infty}\  \sum
\limits_{j=1}^{d(m)} \ |f_{m,j}(r)|^2 \right)^{\frac{p}{2}}\
r^{n-1} \ dr \right)^{\frac{2}{p}}.
\end{eqnarray*}
Thus we have proved that for the operator $T$ defined on
$L^p(\R^+, l^2)$ by
$$Tf(r)=(g_{k,m}^*(f_{m,j})), \ f=(f_{m,j})$$
$\|Tf\|_{L^p(\R^+,l^2)} <\infty$ for each $f\in L^p(\R^+,l^2).$ By
the uniform boundedness principle it follows that $T$ is actually
bounded on $L^p(\R^+,l^2).$ This proves the inequality
$\|F\|_{L^{p,2}(\R^n)} \leq C \|f\|_{L^{p,2}(\R^n)}$ for $p\geq 2$
and hence the theorem follows.

\begin{center}
{\bf Acknowledgments}
\end{center}
The first author is thankful to the National Board for Higher
Mathematics, India, for the financial support. The work of the
second author is supported by J. C. Bose Fellowship from the
Department of Science and Technology (DST) and also by a grant
from UGC via DSA-SAP.


\begin{thebibliography}{99}

\bibitem{Ca} A. P. Calder\'{o}n, \textit{Inequalities for the maximal function relative to
a metric,} Studia Math. \textbf{57} (1976), no. 3, 297-306.

\bibitem{CR} O. Ciaurri and L. Roncal, \textit{The Riesz transform for the Harmonics oscillator in spherical
coordinates,} preprint 2013, arXiv:1304.0702.

\bibitem{CR2} O. Ciaurri and L. Roncal, \textit{Vector-valued extensions for fractional integrals of Laguerre expansions,}
 preprint 2012, arXiv:1304.0702.

\bibitem{DX} C. Dunkl and Y. Xu, \textit{Orthogonal polynomials of several variables.
 Encyclopedia of Mathematics and its Applications,} \textbf{81}, Cambridge University Press, Cambridge, 2001.

\bibitem{Du1}J. Duoandikoetxea, \textit{Extrapolation of weights revisited:
new proofs and sharp bounds,} J. Funct. Anal. \textbf{260} (2011),
no. 6, 1886-1901.

\bibitem{Du2}J. Duoandikoetxea, \textit{Fourier analysis,} Graduate Studies in
Mathematics, 29. American Mathematical Society, Providence, RI,
2001.

\bibitem{KT} R. Kerman and S. Thangavelu, \textit{Weighted inequalities
for semigroups of operators and the norm convergence of Abel means
of certain eigenfunction expansions,} preprint.

\bibitem{MZ} J. Marcinkiewicz and A. Zygmund, \textit{Quelques in\'{e}galit\'{e}s
pour les operations lin\'{e}aires,} Fundamenta Math.,
\textbf{32}(1939), 115-121.

\bibitem{NS} A. Nowak and K. Stempak, \textit{Riesz
transforms for multi-dimensional Laguerre function expansions.}
 Adv. Math. \textbf{215} (2007), no. 2, 642-678.

\bibitem{Ru} J. L. Rubio de Francia, \textit{Vector valued inequalities for
$L^p$ spaces,} Bull. London Math. Soc. \textbf{12} (1980),
211-215.

\bibitem{St} K. Stempak, \textit{Almost everywhere summability of Laguerre
series,} Stud. Math. \textbf{100}(2)(1991), 129-147.

\bibitem{ST} K. Stempak and J. L. Torrea, \textit{Poisson integrals and Riesz transforms for Hermite function
expansions with weights,} J. Funct. Anal. \textbf{202} (2003), no.
2, 443-472.

\bibitem{Th} S. Thangavelu, \textit{Lectures on Hermite and Laguerre expansions,}
Mathematical Notes, 42. Princeton University Press, Princeton, NJ,
1993.
\end{thebibliography}
\end{document}